\documentclass[11pt,reqno]{amsart}
\usepackage{cmap}
\usepackage[T1]{fontenc}
\usepackage[utf8]{inputenc}
\usepackage[english]{babel}
\usepackage{amsfonts,amssymb,amsthm,amsmath}
\usepackage{url}
\usepackage{enumitem}

\usepackage{secdot}

\usepackage{graphicx}
\usepackage{epstopdf}
\usepackage{a4wide}
\usepackage{tikz}
\usepackage{xcolor}
\usepackage[colorlinks=true,linktocpage,pdfpagelabels,
bookmarksnumbered,bookmarksopen]{hyperref}
\definecolor{ForestGreen}{rgb}{0.1,0.6,0.05}
\definecolor{EgyptBlue}{rgb}{0.063,0.1,0.6}
\hypersetup{
	colorlinks=true,
	linkcolor=EgyptBlue,         
	citecolor=ForestGreen,
	urlcolor=olive
}

\usepackage[hyperpageref]{backref}

\newtheorem{theorem}{Theorem}
\newtheorem{proposition}[theorem]{Proposition}
\newtheorem{lemma}[theorem]{Lemma}

\theoremstyle{definition}

\newtheorem{remark}[theorem]{Remark}

\let\OLDthebibliography\thebibliography
\renewcommand\thebibliography[1]{
	\OLDthebibliography{#1}
	\setlength{\parskip}{1pt}
	\setlength{\itemsep}{1pt plus 0.3ex}
}

\numberwithin{equation}{section}
\numberwithin{theorem}{section}
\numberwithin{equation}{section}
\numberwithin{theorem}{section}

\makeatletter
\DeclareRobustCommand{\oominus}{%
	\mathbin{\mathpalette\o@minus@circ\relax}%
}
\newcommand{\o@minus@circ}[2]{%
	\ooalign{$\m@th#1\ominus$\cr$\m@th#1\hspace{.1075em} \circ$\cr}%
}
\makeatother

\title[]{Nonradiality of second eigenfunctions of the fractional Laplacian in a ball}

\author[J.~Benedikt]{Ji\v{r}\'{i} Benedikt}
\address{\newpage
	Department of Mathematics and NTIS, Faculty of Applied Sciences,
	\newline\indent 
	University of West Bohemia, Univerzitn\'i 8, 301 00 Plze\v{n}, Czech Republic
}
\email{benedikt@kma.zcu.cz}

\author[V.~Bobkov]{Vladimir Bobkov}
\address{
	Institute of Mathematics, Ufa Federal Research Centre, RAS,
	\newline\indent 
	Chernyshevsky str. 112, 450008 Ufa, Russia
	\\
	\newline\indent
	Department of Mathematics and NTIS, Faculty of Applied Sciences,
	\newline\indent 
	University of West Bohemia, Univerzitn\'i 8, 301 00 Plze\v{n}, Czech Republic
}
\email{bobkov@matem.anrb.ru}

\author[R.~N.~Dhara]{Raj Narayan Dhara}
\address{Department of Mathematics and Statistics, Faculty of Science
	\newline\indent
	Masaryk University, Kotl\'a\v{r}sk\'a 2, 611 37 Brno, Czech Republic
}
\email{dhara@math.muni.cz}

\author[P.~Girg]{Petr Girg}
\address{
	Department of Mathematics and NTIS, Faculty of Applied Sciences,
	\newline\indent 
	University of West Bohemia, Univerzitn\'i 8, 301 00 Plze\v{n}, Czech Republic
}
\email{pgirg@kma.zcu.cz}

\date{}


\thanks{
	V.~Bobkov was supported in the framework of executing
	the development program of Volga Region Mathematical Center (agreement no.~075-02-2022-888).
	R.~N.~Dhara was supported by Mobility 3.0, Project no.:   CZ.02.2.69/0.0/0.0/16\_027/0008370 and Czech Science Foundation, project GJ19-14413Y}

\begin{document}
	\begin{abstract} 
		Using symmetrization techniques, we show that, for every $N \geq 2$, any second eigenfunction of the fractional Laplacian  in the $N$-dimensional unit ball with homogeneous Dirichlet conditions is nonradial, and hence its nodal set is an equatorial section of the ball. 
	\end{abstract} 
	\maketitle 
	
	\section{Introduction and main result}\label{sec:intro}
	Let $\Omega \subset \mathbb{R}^N$ be a bounded open set, $N \geq 2$.
	For any $s \in (0,1)$, consider the eigenvalue problem
	\begin{equation}\label{eq:D0}
		\left\{
		\begin{aligned}
			(-\Delta)^s u &= \lambda u &&\text{in}~ \Omega,\\
			u &= 0 &&\text{in}~ \mathbb{R}^N \setminus \Omega,
		\end{aligned}
		\right.
	\end{equation}
	where $(-\Delta)^s$ is the fractional Laplacian defined pointwise (for a suitable class of functions, see, e.g., \cite[Section 2]{garofalo}) as
	$$
	(-\Delta)^s u(x) = 
	\frac{s 2^{2s} \Gamma(\frac{N+2s}{2})}{\pi^\frac{N}{2} \Gamma(1-s)}
	\lim_{\varepsilon \to 0+} \int_{\mathbb{R}^N \setminus B(0,\varepsilon)} \frac{u(y)-u(x)}{|y-x|^{N+2s}} \, dy.
	$$
	Let us mention that if $s \to 1-$, then $(-\Delta)^s$ ``tends'' to the usual Laplace operator $-\Delta$, 
	while if $s \to 0+$, then $(-\Delta)^s$ ``tends'' to the identity operator, see, e.g., \cite{BH,Hitch,ST} for rigorous details.
	It is known that the eigenvalues of the problem \eqref{eq:D0} form a nondecreasing sequence 
	$$
	0 < \lambda_1 < \lambda_2 \leq \dots \leq \lambda_k \to +\infty, \quad k \to +\infty,
	$$
	see, e.g. \cite[Proposition 9 (d)]{Servadei2}.
	The first eigenfunction $\varphi_1$ associated with $\lambda_1$ 
	is unique modulo scaling and 
	has a constant sign in $\Omega$, while any eigenfunction associated with $\lambda_k$, $k \geq 2$, is sign-changing, see, e.g., \cite[Theorem 2.8]{BP}.
	We refer to 
	\cite{BP,frank,lindgren,Servadei2}
	for some other fundamental properties of eigenfunctions and eigenvalues of the problem \eqref{eq:D0}. 
	
	Consider now the problem \eqref{eq:D0} in the $N$-dimensional open unit ball $B \subset \mathbb{R}^N$ centred at the origin:
	\begin{equation}\label{eq:D}
		\left\{
		\begin{aligned}
			(-\Delta)^s u &= \lambda u &&\text{in}~ B,\\
			u &= 0 &&\text{in}~ \mathbb{R}^N \setminus B.
		\end{aligned}
		\right.
	\end{equation}
	Note that the first eigenfunction of \eqref{eq:D} is radial, as it follows, e.g., from the results of \cite{DKK}.
	Let us denote by $\lambda_\circledcirc$ the second smallest eigenvalue of \eqref{eq:D} which has an associated radial eigenfunction, 
	and by $\lambda_\ominus$ an eigenvalue whose associated eigenfunction $u$ has an $(N-1)$-dimensional unit ball centred at the origin as its nodal set $\overline{\{x\in B: u(x)=0\}}$.
	It follows from \cite[Proposition 1.1 and Section 3]{DKK} (see also Section \ref{sec:eigenvaluesinball} below) that $\lambda_2 = \min\{\lambda_\ominus, \lambda_\circledcirc\}$.	
	
	The conjecture attributed in \cite{DKK} to \textsc{Kulczycki} asserts that $\lambda_\ominus < \lambda_\circledcirc$, and, consequently, $\lambda_2 = \lambda_\ominus$, for any $N \geq 1$ and $s \in (0,1)$. 
	For the local Laplacian (i.e., $s=1$) this assertion is well known to be true.
	In this classical case and for $N \geq 2$ (the nontrivial case), the assertion follows, for instance, from the interlacing of zeros of appropriate Bessel functions, see, e.g., \cite[Section 15.22]{watson}.
	On the other hand, in the nonlocal setting, the situation is considerably different, which is caused, 
	in particular, by the lack of explicit (i.e., closed form) solutions of the problem \eqref{eq:D}.
	Nevertheless, this conjecture was proved 
	by \textsc{Ba\~{n}uelos \& Kulczycki} \cite{BanKul} in the case $N=1$ and $s=1/2$, and then by \textsc{Dyda, Kuznetsov, \& Kwa\'snicki} \cite{DKK} in two cases: either $1 \leq N \leq 9$ and $s=1/2$, or $N=1,2$ and $s \in (0,1)$.
	Moreover, there is a strong numerical evidence in \cite{DKK} that \textsc{Kulczycki}'s conjecture is true for $1 \leq N \leq 9$ and $s \in (0,1)$.
	Later, the analytic approach of \cite{DKK} was  extended by \textsc{Ferreira} \cite{F} to cover the case $N=3$ and $s \in (0,1)$.
	The proof of \cite{DKK} (and \cite{F}) is based on the establishment of an upper bound $\overline{\lambda}_\ominus$ for $\lambda_\ominus$ and a lower bound $\underline{\lambda}_\circledcirc$ for $\lambda_\circledcirc$ such that
	$$
	\lambda_\ominus \leq \overline{\lambda}_\ominus < \underline{\lambda}_\circledcirc \leq \lambda_\circledcirc,
	$$
	which yields the claimed result. 
	The restrictions on $N$ and $s$ in \cite{DKK,F} are caused by the complexity of construction of the bounds $\overline{\lambda}_\ominus$ and $\underline{\lambda}_\circledcirc$. 
	During the preparation of the manuscript, we have noticed that the conjecture was fully resolved by \textsc{Fall et al} \cite{FFTW} by estimating the Morse index of the second radial eigenfunctions.

	The aim of the present article is to provide another proof of this result following a different strategy from that in \cite{BanKul,DKK,FFTW,F}. 
	Our approach is based on the symmetrization technique called polarization which turned out to be effective in the study of symmetry properties of sign-changing solutions for problems involving local second-order operators of both linear and nonlinear nature, and it follows the ideas developed in \cite{BK} for the $p$-Laplacian. 
	As a remark, we mention that for the $p$-Laplacian, the inequalities $\lambda_2 \leq \lambda_\ominus < \lambda_\circledcirc$ were established in \cite{bendrabgirg} in the planar case $N=2$, and it was proved in \cite{anoop} that $\lambda_2 < \lambda_\circledcirc$ for any $N \geq 2$. Both results yield the nonradiality of the second eigenfunctions of the $p$-Laplacian in a ball (see also \cite{BK}), while the exact shape of the corresponding nodal sets remain obscure.
	
	Although our approach is inspired by \cite{BK}, certain steps in the arguments of \cite{BK} are hardly applicable directly to the nonlocal settings.
	In particular, it is known  
	for the classical Laplacian and $p$-Laplacian that the restriction of any higher eigenfunction to some of its nodal domains is the first eigenfunction of that nodal domain.
	This fact combined with the Hopf maximum principle 
	provides information about the derivative of a second eigenfunction on the boundary of the nodal domain, which was used in \cite{BK}.
	Unfortunately, as far as we know, the same approach does not apply in the nonlocal case of the fractional Laplacian (see the discussion in  \cite[Section 6]{lindgren}), and, moreover, the Courant nodal domain property has not been established even for second eigenfunctions of  the fractional Laplacian on a bounded domain.
	Instead, we significantly use the linearity of the problem \eqref{eq:D} in combination with the information on the structure of its spectrum obtained in \cite{DKK} to prove the following result.
	\begin{theorem}\label{thm}
		Let $N \geq 2$ and $s \in (0,1)$. Then $\lambda_2 = \lambda_\ominus < \lambda_\circledcirc$.
	\end{theorem}
	
	Let us outline the idea of the proof of Theorem \ref{thm} given in Section \ref{sec:proof}. 
	Arguing by contradiction, we suppose that there exists a radial eigenfunction corresponding to $\lambda_2$. 
	Then, properly symmetrizing this eigenfunction by means of polarization (see Section \ref{sec:polarization}), we construct a \textit{nonradial} second eigenfunction of \eqref{eq:D} using a result of Section \ref{sec:second}. 
	Finally, employing the linearity of $(-\Delta)^s$ and the essence of \cite[Proposition 1.1]{DKK} described in Section \ref{sec:eigenvaluesinball}, we show that the nodal set of this alleged eigenfunction has certain contradictory properties. 
	Let us remark that in the proof of Theorem \ref{thm} we use only minimal regularity assumptions on eigenfunctions of \eqref{eq:D} given by \cite[Proposition 4]{Servadei3} and \cite[Theorem 1.1]{ROS} and do not require any information on the unique continuation property.
	The details are as follows.

	\section{Preliminaries}
	In this section, we establish the main ingredients needed for the proof of Theorem~\ref{thm}.  
	Throughout the text, we always denote $\mathbb{N} = \{1,2,\dots\}$ and $\mathbb{N}_0 = \mathbb{N} \cup \{0\}$.
	\subsection{Functional framework}
	Let $\Omega\subset \mathbb{R}^N$, $N\ge 2$, be a bounded open set. 
	For $s\in (0,1)$, we denote by $H^s(\mathbb{R}^N)$ the standard fractional Sobolev space 
	\begin{equation}\label{eq:defH}
		H^s(\mathbb{R}^N) 
		:= 
		\left\{
		u \in L^2(\mathbb{R}^N):~ [u]_{H^{s}(\mathbb{R}^N)}<+\infty
		\right\}
	\end{equation}
	endowed with the norm
	$$
	\|u\|_{H^s(\mathbb{R}^N)} := \|u\|_{L^2(\mathbb{R}^N)} + [u]_{H^{s}(\mathbb{R}^N)},
	$$
	where $\|u\|_{L^2(\mathbb{R}^N)} := \left(\int_{\mathbb{R}^N} |u|^2 \, dx\right)^{1/2}$ and  
	\begin{align*}
		[u]_{H^{s}(\mathbb{R}^N)}
		:=
		\left(\iint_{\mathbb{R}^{N}\times\mathbb{R}^{N}}\frac{|u(x)-u(y)|^{2}}{|x-y|^{N+2s}}\, dxdy
		\right)^{1/2}
	\end{align*}
	is the Gagliardo seminorm.
	Consider now the subspace
	$$
	X_0^s(\Omega) 
	:= 
	\left\{
	u \in H^s(\mathbb{R}^N):~ u=0 ~\text{a.e. in}~ \mathbb{R}^N \setminus \Omega
	\right\}.
	$$
	It is known (see, e.g, \cite[Lemma 7]{Servadei2}) that $X_0^s(\Omega)$ is a Hilbert space with the scalar product 
	\begin{align*}
		\langle u,v \rangle 
		:= 
		\frac{c_{N,s}}{2}
		\iint_{\mathbb{R}^{N}\times\mathbb{R}^{N}}\dfrac{(u(x)-u(y))(v(x)-v(y))}{|x-y|^{N+2s}}\, dxdy
	\end{align*}
	and the associated norm
	\begin{align*}
		\|u\|_{X_0^s(\Omega)}
		:=
		\langle u,u \rangle^{1/2}
		=
		\sqrt{\frac{c_{N,s}}{2}}\, [u]_{H^{s}(\mathbb{R}^N)},
	\end{align*}
	where
	\begin{align*}
		c_{N,s}=\frac{s 2^{2s} \Gamma(\frac{N+2s}{2})}{\pi^\frac{N}{2} \Gamma(1-s)}.
	\end{align*}
	
	We will deal with the weak formulation of the problem \eqref{eq:D0}.
	That is, by a weak solution of \eqref{eq:D0} we mean a function $u \in X_0^s(\Omega)$ such that
	\begin{equation}\label{eq:weak0}
		\left<u,\xi\right> 
		= 
		\lambda \int_\Omega u \,\xi \, dx 
		\quad \text{for all}~ \xi \in X_0^s(\Omega).
	\end{equation}
	We say that $\lambda \in \mathbb{R}$ is an eigenvalue of \eqref{eq:D0} provided there exists a nontrivial function $u \in X_0^s(\Omega)$, called an eigenfunction, such that \eqref{eq:weak0} is satisfied.
	Recall that the set of all eigenvalues is discrete and consists of the sequence 
	$$
	0 < \lambda_1 < \lambda_2 \leq \dots \leq \lambda_k \to +\infty, \quad \mathbb{N}\ni k \to +\infty,
	$$
	see, e.g., \cite[Proposition 9 (d)]{Servadei2}.
	In particular, any $\lambda_k$ has a finite multiplicity and 
	\begin{equation}\label{eq:weak}
		\left<\varphi_k,\xi\right>
		= 
		\lambda_k \int_\Omega \varphi_k \xi \, dx 
		\quad \text{for all}~ \xi \in X_0^s(\Omega),
	\end{equation}
	where $\varphi_k \in X_0^s(\Omega) \setminus \{0\}$ stands for an eigenfunction associated with $\lambda_k$.
	If $\Omega$ is a bounded Lipschitz domain satisfying the exterior ball condition, then $\varphi_k \in C^{0,s}(\overline{\Omega})$, see \cite[Proposition 1.1]{ROS2} in combination with \cite[Proposition 4]{Servadei3}.
	
	Let us also recall that $\varphi_1$ has a constant sign in $\Omega$, while any higher eigenfunction $\varphi_k$ is sign-changing, i.e., $\varphi_k^\pm \not\equiv 0$ in $\Omega$ for any $k \geq 2$.
	Hereinafter, for a function $v$ we denote $v^+ := \max\{v,0\}$ and $v^- := \max\{-v,0\}$, and hence $v=v^+-v^-$. 
	
	\subsection{Second eigenvalue}\label{sec:second}
	The second eigenvalue $\lambda_2$ of \eqref{eq:D0} has the following variational characterization (see, e.g., \cite[Proposition 9 (d)]{Servadei2}):
	\begin{equation}\label{eq:deflambda2}
		\lambda_2 
		= 
		\inf 
		\left\{
		\frac{\left<u,u\right>}{\int_\Omega u^2 \, dx}:~
		u \in X_0^s(\Omega)\setminus \{0\},~ \left<u,\varphi_1\right>=0
		\right\}.
	\end{equation}
	Noting that $\varphi_2^\pm \in X_0^s(\Omega)$, 
	we use $\varphi_2^+$ and $\varphi_2^-$ as test functions in \eqref{eq:weak} with $k=2$ to get
	\begin{equation}\label{eq:-1}
		\lambda_2 \int_\Omega (\varphi_2^+)^2 \, dx
		=
		\left<\varphi_2,\varphi_2^+\right>
		\quad \text{and} \quad
		\lambda_2 \int_\Omega (\varphi_2^-)^2 \, dx
		=
		-\left<\varphi_2,\varphi_2^-\right>.
	\end{equation}
	
	We provide the following result.
	\begin{lemma}\label{lem:1}
		Assume that there exists a function $v \in X_0^s(\Omega)$ such that $v^\pm \not\equiv 0$  in $\Omega$ and 
		\begin{equation}
			\label{eq:1}
			\lambda_2 \int_\Omega (v^+)^2 \, dx 
			\geq 
			\left<v,v^+\right>
			\quad \text{and}\quad 
			\lambda_2 \int_\Omega (v^-)^2 \, dx 
			\geq 
			-\left<v,v^-\right>.
		\end{equation}
		Then $v$ is an eigenfunction  associated with the second eigenvalue $\lambda_2$ of \eqref{eq:D0}, and the equalities hold in \eqref{eq:1}.
	\end{lemma}
	\begin{proof}
		First, we show that there exists $\alpha_0 > 0$ satisfying $\left<v^+ - \alpha_0 v^-,  \varphi_1\right>=0$.
		For any $\alpha \in \mathbb{R}$, we have
		\begin{equation}\label{eq:lemx1}
			\left<v^+ - \alpha v^-, \varphi_1\right>
			=
			\left<v^+, \varphi_1\right>
			-\alpha
			\left<v^-, \varphi_1\right>.
		\end{equation}
		Using $v^+$ and $v^-$ as test functions in the weak formulation \eqref{eq:weak} with $k=1$, 
		we get
		\begin{equation}\label{eq:lemx2}
			\left<v^\pm,\varphi_1\right>
			=
			\lambda_1
			\int_\Omega v^\pm \varphi_1 \, dx > 0,
		\end{equation}
		where the last inequality follows from the facts that $\varphi_1>0$ in $\Omega$ and $v^\pm$ are nontrivial and nonnegative. 
		Hence, the existence of $\alpha_0>0$ such that $\left<v^+ - \alpha_0 v^-,  \varphi_1\right>=0$ easily follows from \eqref{eq:lemx1} and \eqref{eq:lemx2}.
		
		Thus, $v^+ - \alpha_0 v^- \in X_0^s(\Omega)\setminus \{0\}$ is an admissible function for the variational characterization \eqref{eq:deflambda2} of $\lambda_2$, which yields
		\begin{equation}\label{eq:lambda2<}
			\lambda_2 
			\leq 
			\frac{\left<v^+ - \alpha_0 v^-,v^+ - \alpha_0 v^-\right>}
			{\int_\Omega (v^+)^2 \, dx + \alpha_0^2 \int_\Omega (v^-)^2 \, dx}.
		\end{equation}
		On the other hand, multiplying the second inequality in \eqref{eq:1} by $\alpha_0^2$ and adding it to the first one, we obtain
		$$
		\lambda_2
		\geq 
		\frac{\left<v, v^+ - \alpha_0^2 v^-\right>}{\int_\Omega (v^+)^2 \, dx + \alpha_0^2 \int_\Omega (v^-)^2 \, dx}.
		$$
		Observing that 
		$$
		v(x)-v(y) = (v^+(x)-v^+(y)) - (v^-(x)-v^-(y))
		$$
		and
		\begin{equation}\label{eq:leq0}
		(v^+(x)-v^+(y))(v^-(x)-v^-(y)) \leq 0,
		\end{equation}
		we use the inequality $1+\alpha_0^2 \geq 2\alpha_0$ to get
		\begin{align}
			\notag
			&(v(x)-v(y))(v^+(x)-v^+(y)) - \alpha_0^2 (v(x)-v(y))(v^-(x)-v^-(y))\\
			\notag
			&=
			(v^+(x)-v^+(y))^2
			-(1+\alpha_0^2)(v^+(x)-v^+(y))(v^-(x)-v^-(y))
			+
			\alpha_0^2 (v^-(x)-v^-(y))^2\\
			\notag
			&\geq
			(v^+(x)-v^+(y))^2
			-2\alpha_0(v^+(x)-v^+(y))(v^-(x)-v^-(y))
			+
			\alpha_0^2 (v^-(x)-v^-(y))^2
			\\
			\label{eq:lemsecon2}
			&=
			\left((v^+(x)-v^+(y))-\alpha_0(v^-(x)-v^-(y))\right)^2.
		\end{align}
		Therefore, recalling that
		\begin{align*}
			&\left<v, v^+ - \alpha_0^2 v^-\right>
			\\
			&=
			\frac{c_{N,s}}{2}
			\int_{\mathbb{R}^N}\int_{\mathbb{R}^N} \frac{(v(x)-v(y))(v^+(x)-v^+(y)) - \alpha_0^2 (v(x)-v(y))(v^-(x)-v^-(y))}{|x-y|^{N+2s}} \, dx\,dy,
		\end{align*}
		we deduce from \eqref{eq:lemsecon2} the inequality 
		$$
		\left<v, v^+ - \alpha_0^2 v^-\right> \geq \left<v^+ - \alpha_0 v^-,v^+ - \alpha_0 v^-\right>,
		$$
		which yields
		\begin{equation}\label{eq:lambda2>}
			\lambda_2 
			\geq 
			\frac{\left<v^+ - \alpha_0 v^-,v^+ - \alpha_0 v^-\right>}
			{\int_\Omega (v^+)^2 \, dx + \alpha_0^2 \int_\Omega (v^-)^2 \, dx}.
		\end{equation}
		We conclude from \eqref{eq:lambda2<} and \eqref{eq:lambda2>} that the equality in \eqref{eq:lambda2<} holds, and hence 
		$v^+ - \alpha_0 v^-$ is a minimizer of the right-hand side of the variational characterization \eqref{eq:deflambda2} of $\lambda_2$. 
		At the same time, arguing as in the proof of \cite[Proposition 9]{Servadei2} (see, more precisely, the proof of Claim~1 on p.~2128 in \cite{Servadei2}), one can show that any minimizer of \eqref{eq:deflambda2} is a second eigenfunction of \eqref{eq:D0}.		
		Therefore, $v^+ - \alpha_0 v^-$ is an eigenfunction associated with $\lambda_2$.
		
		We will show that $\alpha_0=1$.
		Thanks to \eqref{eq:lambda2<} and \eqref{eq:lambda2>}, the equality must hold in \eqref{eq:lemsecon2}.
		Taking, for instance, $x \in \text{supp}\, v^-$ and $y \in \text{supp}\, v^+$, we get the strict inequality in \eqref{eq:leq0}, and hence the equality in \eqref{eq:lemsecon2} enforces $\alpha_0=1$.
		Finally, in view of \eqref{eq:-1}, we see that the inequalities in \eqref{eq:1} are actually equalities.
	\end{proof}
	
	\begin{remark}
	As a direct corollary of Lemma \ref{lem:1}, we have the following alternative characterization of the second eigenvalue $\lambda_2$ (cf.\ \eqref{eq:deflambda2} and the characterizations provided in \cite{BP}):
	$$
	\lambda_2 
	=
	\inf 
	\left\{
	\max
	\left\{
	\frac{\left<u,u^+\right>}{\int_\Omega (u^+)^2 \, dx},
	\frac{-\left<u,u^-\right>}{\int_\Omega (u^-)^2 \, dx}
	\right\}:~
	u \in X_0^s(\Omega), ~ u^\pm \not\equiv 0 \text{ in } \Omega
	\right\}. 
	$$
	\end{remark}

	\subsection{Polarization}\label{sec:polarization}
	
	To find a function $v$ which would satisfy the assumptions of Lemma \ref{lem:1}, we will use the symmetrization technique called \textit{polarization}, see, e.g., \cite{BE,brocksol}.
	Taking any $a \in \mathbb{R}$, consider the hyperplane $H_a := \{x \in \mathbb{R}^N: x_1 = a \}$ where $x := (x_1,x_2,\dots,x_N)$, and let $\bar{x} := (2a - x_1, x_2, \dots, x_N)$ be the reflection of $x$ with respect to $H_a$. 
	Denote the half-spaces separated by $H_a$ by
	\begin{equation*}
		\Sigma_a^- := \{ x \in \mathbb{R}^N:~ x_1 \leq a \}
		\quad \text{and} \quad 
		\Sigma_a^+ := \{ x \in \mathbb{R}^N:~ x_1 \geq a \} .
	\end{equation*} 
	Let us fix any $u \in H^{s}(\mathbb{R}^N)$ and define the polarization of $u$ with respect to $H_a$ as a new function $P_a u$ as
	\begin{equation}
		\label{eq:polarization}
		(P_a u)(x) = 
		\left\{
		\begin{aligned}
			&\min \{u(x), u(\bar{x})\}, &&x \in \Sigma_a^+,\\
			&\max \{u(x), u(\bar{x})\}, &&x \in \Sigma_a^-.
		\end{aligned}
		\right.
	\end{equation}
	It is not hard to see that
	\begin{equation}\label{eq:polrefl}
		(P_a(-u))(x) = -(P_a u)(\bar{x}) 
		\quad \text{for any}~ x \in \mathbb{R}^N.
	\end{equation}
	Moreover, there holds
	\begin{equation}\label{eq:3}
		[u]_{H^{s}(\mathbb{R}^N)} \geq [P_a u]_{H^{s}(\mathbb{R}^N)},
		\quad
		\|u^\pm\|_{L^2(\mathbb{R}^N)} = \|(P_a u)^\pm\|_{L^2(\mathbb{R}^N)},
	\end{equation}
	see, e.g., \cite[p.~4818]{BE} or Lemma \ref{lem:2} below for the inequality.
	In particular, it is seen from the definition \eqref{eq:defH} of $H^{s}(\mathbb{R}^N)$ that $P_a u \in H^{s}(\mathbb{R}^N)$.
	The following extension of the inequality in \eqref{eq:3} will be crucial for us.
	\begin{lemma}\label{lem:2}
		Let $u \in H^{s}(\mathbb{R}^N)$ and $a \in \mathbb{R}$. 
		Then
		\begin{equation}\label{eq:5}
			\left<u,u^+\right> \geq \left<P_a u,(P_a u)^+\right>
			\quad \text{and} \quad 
			-\left<u,u^-\right> \geq -\left<P_a u,(P_a u)^-\right>.
		\end{equation}
	\end{lemma}
	\begin{proof}
		Throughout the proof, we denote $v:=P_a u$.
		We start by establishing the first inequality in \eqref{eq:5}.
		Let us denote $F(s,t) := (s-t)(s^+-t^+)$. Recalling that $\mathbb{R}^N = \Sigma_a^+ \cup \Sigma_a^-$ and using a change of variables, we have
		\begin{align*}
			\int_{\mathbb{R}^N}\int_{\mathbb{R}^N} \frac{F(u(x),u(y))}{|x-y|^{N+2s}} \, dx\,dy
			&=
			\int_{\Sigma_a^+}\int_{\Sigma_a^+} \frac{F(u(x),u(y))}{|x-y|^{N+2s}} \, dx\,dy
			+
			\int_{\Sigma_a^+}\int_{\Sigma_a^+} \frac{F(u(\bar{x}),u(y))}{|\bar{x}-y|^{N+2s}} \, dx\,dy
			\\
			&+
			\int_{\Sigma_a^+}\int_{\Sigma_a^+} \frac{F(u(x),u(\bar{y}))}{|x-\bar{y}|^{N+2s}} \, dx\,dy
			+
			\int_{\Sigma_a^+}\int_{\Sigma_a^+} \frac{F(u(\bar{x}),u(\bar{y}))}{|\bar{x}-\bar{y}|^{N+2s}} \, dx\,dy,
		\end{align*}
		and the same holds with $F(v(x),v(y))$.	
		Thus, to obtain the claimed first inequality in \eqref{eq:5}, it is sufficient to prove that
		\begin{align}
			\notag
			\frac{F(u(x),u(y))}{|x-y|^{N+2s}} &+ \frac{F(u(\bar{x}),u(y))}{|\bar{x}-y|^{N+2s}} +  \frac{F(u(x),u(\bar{y}))}{|x-\bar{y}|^{N+2s}} + \frac{F(u(\bar{x}),u(\bar{y}))}{|\bar{x}-\bar{y}|^{N+2s}}
			\\
			\label{eq:8}
			&\geq
			\frac{F(v(x),v(y))}{|x-y|^{N+2s}} + \frac{F(v(\bar{x}),v(y))}{|\bar{x}-y|^{N+2s}} +  \frac{F(v(x),v(\bar{y}))}{|x-\bar{y}|^{N+2s}} + \frac{F(v(\bar{x}),v(\bar{y}))}{|\bar{x}-\bar{y}|^{N+2s}}
		\end{align}
		for all $x, y \in \Sigma_a^+$. 
		To prove \eqref{eq:8}, let us first notice that for every $x, y \in \Sigma_a^+$, we have
		\begin{equation}\label{eq:7}
			\frac{1}{|x-y|^{N+2s}} = \frac{1}{|\bar{x}-\bar{y}|^{N+2s}} \geq \frac{1}{|\bar{x}-y|^{N+2s}} = \frac{1}{|x-\bar{y}|^{N+2s}},
		\end{equation}
		as it can be deduced from the definition of $\bar{x}$ and $\bar{y}$. 
		
		We prove \eqref{eq:8} in the following four separate cases:
		\begin{enumerate}
			\item Assume that $x, y \in \Sigma_a^+$ are such that 
			$$
			u(\bar{x}) \geq u(x) 
			\quad \text{and} \quad 
			u(\bar{y}) \geq u(y).
			$$
			Then, by the definition \eqref{eq:polarization} of polarization, we have $v = u$, and hence \eqref{eq:8} becomes an equality.
			
			\item Assume that $x, y \in \Sigma_a^+$ are such that 
			$$
			u(\bar{x}) \leq u(x) 
			\quad \text{and} \quad 
			u(\bar{y}) \leq u(y).
			$$
			Then $v(x)=u(\bar{x})$, $v(\bar{x})=u(x)$, $v(y)=u(\bar{y})$, $v(\bar{y})=u(y)$, and hence the left-hand side of \eqref{eq:8} can be written as
			\begin{align*}
				\frac{F(u(x),u(y))}{|x-y|^{N+2s}} &+ \frac{F(u(\bar{x}),u(y))}{|\bar{x}-y|^{N+2s}} + \frac{F(u(x),u(\bar{y}))}{|x-\bar{y}|^{N+2s}} + \frac{F(u(\bar{x}),u(\bar{y}))}{|\bar{x}-\bar{y}|^{N+2s}}
				\\
				& \stackrel{\eqref{eq:7}}{=}
				\frac{F(v(\bar{x}),v(\bar{y}))}{|\bar{x}-\bar{y}|^{N+2s}}+ \frac{F(v(\bar{x}),v(y))}{|\bar{x}-y|^{N+2s}} +  \frac{F(v(x),v(\bar{y}))}{|x-\bar{y}|^{N+2s}} + \frac{F(v(x),v(y))}{|x-y|^{N+2s}}.
			\end{align*}
			That is, \eqref{eq:8} turns out to be an equality in this case as well.
			
			\item Assume that $x, y \in \Sigma_a^+$ are such that 
			$$
			u(\bar{x}) \geq u(x) 
			\quad \text{and} \quad 
			u(\bar{y}) \leq u(y).
			$$
			Then $v(x)=u(x)$, $v(\bar{x})=u(\bar{x})$, $v(y)=u(\bar{y})$, $v(\bar{y})=u(y)$, and hence \eqref{eq:8} can be written as
			\begin{align}
			\notag
				\frac{F(u(x),u(y))}{|x-y|^{N+2s}} &+ \frac{F(u(\bar{x}),u(y))}{|\bar{x}-y|^{N+2s}} +  \frac{F(u(x),u(\bar{y}))}{|x-\bar{y}|^{N+2s}} + \frac{F(u(\bar{x}),u(\bar{y}))}{|\bar{x}-\bar{y}|^{N+2s}}
				\\
				\label{eq:FFFF}
				&\geq	
				\frac{F(u(x),u(\bar{y}))}{|x-y|^{N+2s}} + \frac{F(u(\bar{x}),u(\bar{y}))}{|\bar{x}-y|^{N+2s}} +  \frac{F(u(x),u(y))}{|x-\bar{y}|^{N+2s}} + \frac{F(u(\bar{x}),u(y))}{|\bar{x}-\bar{y}|^{N+2s}}.
			\end{align}
			In view of \eqref{eq:7}, we have
			$$
			\frac{1}{|x-y|^{N+2s}} 
			-
			\frac{1}{|x-\bar{y}|^{N+2s}}
			= 
			\frac{1}{|\bar{x}-\bar{y}|^{N+2s}} 
			-
			\frac{1}{|\bar{x}-y|^{N+2s}} \geq 0,
			$$
			and hence, by rearranging the terms in   \eqref{eq:FFFF}, we see that \eqref{eq:FFFF} is satisfied provided
			\begin{align*}
				F(u(x),u(y)) - F(u(\bar{x}),u(y)) - F(u(x),u(\bar{y})) + F(u(\bar{x}),u(\bar{y}))
				\geq 0.
			\end{align*}
			Rewriting $F$ in the original form, opening the brackets, and making standard simplifications, we arrive at
			\begin{align*}
				\notag
				&F(u(x),u(y)) - F(u(\bar{x}),u(y)) - F(u(x),u(\bar{y})) + F(u(\bar{x}),u(\bar{y}))\\
				&=\underbrace{(u(x)-u(\bar{x}))}_{\leq 0} \underbrace{(u^+(\bar{y})-u^+(y))}_{\leq 0} + \underbrace{(u^+(x)-u^+(\bar{x}))}_{\leq 0} \underbrace{(u(\bar{y})-u(y))}_{\leq 0} \geq 0.
			\end{align*}
			Hence, the desired inequality \eqref{eq:8} holds true.
			\item Assume that $x, y \in \Sigma_a^+$ are such that 
			$$
			u(\bar{x}) \leq u(x) 
			\quad \text{and} \quad 
			u(\bar{y}) \geq u(y).
			$$
			Switching the notations $x \leftrightarrow y$, we arrive at the previous case, and hence \eqref{eq:8} is satisfied.
		\end{enumerate}
		Thus, we have shown that \eqref{eq:8} is satisfied for all $x, y \in \Sigma_a^+$, which proves the first inequality in \eqref{eq:5}.
		
		In order to justify the second inequality in \eqref{eq:5}, we first write
		\begin{equation}\label{eq:pol2}
			-\left<u,u^-\right> = \left<-u,(-u)^+\right> \geq \left<P_a (-u),(P_a (-u))^+\right>,
		\end{equation}
		where the inequality is given by the first inequality in \eqref{eq:5} with $-u$ instead of $u$.
		Using \eqref{eq:polrefl} and the fact that $|x-y|=|\bar{x}-\bar{y}|$, we see that
		\begin{align*}
			&\left<P_a (-u),(P_a (-u))^+\right> 
			\\
			&=
			\frac{c_{N,s}}{2}
			\iint_{\mathbb{R}^{N}\times\mathbb{R}^{N}}\frac{(P_a (-u)(x)-P_a (-u)(y))((P_a (-u)(x))^+-(P_a (-u)(y))^+)}{|x-y|^{N+2s}}\, dxdy
			\\
			&=
			\frac{c_{N,s}}{2}
			\iint_{\mathbb{R}^{N}\times\mathbb{R}^{N}}\frac{(-P_a u(\bar{x})+P_a u(\bar{y}))((-P_a u(\bar{x}))^+-(-P_a u(\bar{y}))^+)}{|\bar{x}-\bar{y}|^{N+2s}}\, dxdy
			\\
			&=
			-\frac{c_{N,s}}{2}
			\iint_{\mathbb{R}^{N}\times\mathbb{R}^{N}}\frac{(P_a u(\bar{x})-P_a u(\bar{y}))((P_a u(\bar{x}))^--(P_a u(\bar{y}))^-)}{|\bar{x}-\bar{y}|^{N+2s}}\, dxdy
			\\
			&=
			-\frac{c_{N,s}}{2}
			\iint_{\mathbb{R}^{N}\times\mathbb{R}^{N}}\frac{(P_a u(\bar{x})-P_a u(\bar{y}))((P_a u(\bar{x}))^--(P_a u(\bar{y}))^-)}{|\bar{x}-\bar{y}|^{N+2s}}\, d\bar{x}d\bar{y}
			\\
			&= 
			-\left<P_a u,(P_a u)^-\right>,
		\end{align*}
		where we made the change of variables $\bar{x} = \bar{x}(x)$ and $\bar{y} = \bar{y}(y)$ which is a linear isometry.
		Combining this with \eqref{eq:pol2}, we finish the proof.
	\end{proof}

	\subsection{Eigenvalues in the ball}\label{sec:eigenvaluesinball}
	The results of three previous subsections were formulated for the problem \eqref{eq:D0} in a general $\Omega$.
	In this subsection, we discuss certain specific properties of the problem \eqref{eq:D} in the  ball. 
	
	Let us denote by $\{\varphi_{N,n}(|x|)\}_{n \in \mathbb{N}}$ and $\{\lambda_{N,n}\}_{n \in \mathbb{N}}$ the sequences of all \textit{radial} basis eigenfunctions and corresponding eigenvalues for the problem \eqref{eq:D} in the  $N$-dimensional open unit ball $B$. 
	We assume that $\{\lambda_{N,n}\}_{n \in \mathbb{N}}$ is naturally arranged in the nondecreasing order with respect to $n$, that is, for a fixed dimension $N$,
	\begin{equation}\label{eq:ord-n}
		\lambda_{N,1} 
		< 
		\lambda_{N,2}
		\leq
		\lambda_{N,3}
		\leq
		\dots,
	\end{equation}
	where the first inequality is strict since, evidently, $\lambda_1 = \lambda_{N,1}$, and $\lambda_1$ is simple (see, e.g., \cite[Proposition 9 (c)]{Servadei2}).
	Moreover, each $\lambda_{N,k}$ has a finite multiplicity and $\lambda_{N,k} \to +\infty$ as $k \to +\infty$.

	Let us also denote by $H_{l}$ the space of all homogeneous harmonic polynomials in $N$ variables and of degree $l \in \mathbb{N}_0$.
	It is known that the dimension of $H_l$ is given by the formula
	$$
	M_{l} := \text{dim}\, H_l
	=
	\binom{l+N-1}{N-1} - \binom{l+N-3}{N-1},
	$$
	see, e.g., \cite[Proposition 5.8]{axler}.
	Notice that $M_0 = 1$ and $M_1 = N$.
	Denote by $\{V_{l,m}\}_{m=1}^{M_l}$ an orthogonal basis of $H_l$.
	In particular, the basis of $H_1$ can be chosen as $\{x_1,\dots,x_N\}$. 
	
	The structure of eigenfunctions and eigenvalues of the problem \eqref{eq:D} in $B$ is described in \cite[Proposition 1.1]{DKK}.
	For reader's convenience, we recall its statement below.  
	Notice that in contrast with \cite{DKK} we start the index counting from $1$ rather than from $0$ to preserve the uniformity of notations.
	\begin{proposition}{\cite[Proposition 1.1]{DKK}}\label{prop:spectrum}
		The functions $V_{l,m}(x) \varphi_{N+2l,n}(|x|)$ with $l \in \mathbb{N}_0$, $1 \leq m \leq M_{l}$, and $n \in \mathbb{N}$ form a complete orthogonal system of eigenfunctions of \eqref{eq:D} in the unit $N$-dimensional ball $B$.
		The eigenfunction  $V_{l,m}(x) \varphi_{N+2l,n}(|x|)$ corresponds to the eigenvalue $\lambda_{N+2l,n}$.
	\end{proposition}
	
	As a consequence of Proposition \ref{prop:spectrum}, we have $\{\lambda_k\}_{k \in \mathbb{N}} = \{\lambda_{N+2l,n}\}_{l \in \mathbb{N}_0, n \in \mathbb{N}}$, and any  $\lambda_{N+2l,n}$ is counted in the spectrum of \eqref{eq:D} with the multiplicity $M_l$.
	In particular, if $\lambda_k = \lambda_{N+2l,n}$ for some $k,l,n$, then $\lambda_k$ has the multiplicity 
	\textit{at least} $M_{l}$.
	Let us remark that in the local case $s=1$ the multiplicity of such  $\lambda_k$ is precisely $M_{l}$. This fact is known as Bourget's hypothesis (see, e.g., \cite[Section 15.28]{watson}), and its proof, based on fine properties of Bessel functions, can be found in \cite{Siegel}. 
	Up to our knowledge, the same problem for the fractional Laplacian is open (except for the case of $\lambda_1$ and $\lambda_2$). 
	
	It was proved in \cite[p.~509]{DKK} that 
	\begin{equation}\label{eq:ord-N}
		\lambda_{N,1} 
		< 
		\lambda_{N+1,1}
		<
		\lambda_{N+2,1}
		<
		\dots
	\end{equation}
	Therefore, recalling that $\lambda_1 = \lambda_{N,1}$ and $\lambda_1$ is simple, we deduce from the orderings \eqref{eq:ord-n} and \eqref{eq:ord-N} that 
	$$
	\lambda_2 
	= 
	\min\{\lambda_{N+2,1},\lambda_{N,2}\}
	\equiv
	\min\{\lambda_\ominus,\lambda_\circledcirc\}.
	$$
	Here, the schematic graphical notations $\lambda_\ominus \equiv \lambda_{N+2,1}$ and $\lambda_\circledcirc \equiv \lambda_{N,2}$ 
	are consistent with Proposition \ref{prop:spectrum} and were introduced in Section \ref{sec:intro}.
	
	Finally, we will need the following result.
	\begin{lemma}\label{lem:sign}
		Let $u \in X_0^s(B) \setminus \{0\}$ be a radial eigenfunction of \eqref{eq:D}. 
		Then there exists $\rho \in (0,1)$ such that either $u>0$ or $u<0$ in the spherical shell with inner radius $\rho$ and outer radius $1$.
	\end{lemma}
	\begin{proof}
		Suppose, by contradiction, that there exist a radial eigenfunction $u$ associated with some eigenvalue $\lambda_k$ ($k \geq 2$) and a sequence $\{\rho_n\}_{n \in \mathbb{N}} \subset (0,1)$ satisfying $\rho_n \nearrow 1$ such that $u(x) = 0$ for any $x \in B$ with $|x| = \rho_n$. 
		Letting $\delta(x)=\text{dist}(x,\partial B)$, we have $u(x)/ \delta^{s}(x)=0$ for any $x \in B$ with $|x| = \rho_n$.		
		We know from \cite[Theorem 1.1]{ROS} that 
		$u/ \delta^{s}|_{B}$ admits a continuous extension to $\overline{B}$, denoted by $u/ \delta^{s}$, which implies that $u/ \delta^{s}=0$ on $\partial B$.
		Hence, applying the Pohozaev identity from \cite[Theorem 1.1]{ROS}, we deduce that $u$ satisfies
		$$
		(2s-N) \lambda_k \int_B u^2 \, dx + N  \lambda_k \int_B u^2 \, dx = 0,
		$$
		which yields $u \equiv 0$ in $B$, a contradiction.
	\end{proof}

	\section{Proof of Theorem \ref{thm}}\label{sec:proof}
	Throughout the proof, we denote by $A_{r_1,r_2}$ the open spherical shell centred at the origin and having the inner radius $r_1$ and outer radius $r_2$, where $r_2>r_1>0$.
	
	Suppose, by contradiction, that $\lambda_2 = \lambda_\circledcirc$. 
	Let $u \in X_0^s(B) \setminus \{0\}$ \label{choice-of-u} be an arbitrary radially symmetric second eigenfunction of \eqref{eq:D}. 
	Let us introduce
	\begin{equation}\label{eq:r}
		r = r(u) 
		:=
		\inf 
		\{\varrho \in [0,1]:\, 
		u(x) \geq 0 \text{ for any } x \text{ such that } |x| \in [\varrho,1]\}.
	\end{equation}
	In view of Lemma \ref{lem:sign} and since $u$ is sign-changing, we can assume, without loss of generality, that $u>0$ in $A_{\rho,1}$ and $u=0$ on $\partial B_\rho$ for some $\rho \in (0,1)$, that is, $r \in (0,\rho]$. 
	Such $\rho$ will be used below.
	
	Let us show that  $P_a u \in X_0^s(B)$ for any $a \in (0, \frac{1-r}{2})$.
	To this end, in view of the inequality in \eqref{eq:3}, it is enough to prove that $(P_a u)(x) = 0$ for any $x \in \mathbb{R}^N \setminus B$.    
	Assume first that $x \in (\mathbb{R}^N \setminus B) \cap \Sigma_a^-$, i.e., $|x| \geq 1$ and $x_1 \leq a$. 
	In particular, we have $u(x)=0$ for such $x$.
	Then $\bar{x}$, the reflection of $x$ with respect to the hyperplane $H_a$, satisfies 
	$$
	|\bar{x}|^2 
	= 
	(2a-x_1)^2 + x_2^2 + \dots + x_N^2
	=
	|x|^2 + 4a(a-x_1) \geq |x|^2 \geq 1,
	$$
	that is, $\bar{x} \in (\mathbb{R}^N \setminus B) \cap \Sigma_a^+$, and thus $u(\bar{x}) = 0$.
	This yields $(P_a u)(x) = \max\{u(x), u(\bar{x})\}  = 0$. 
	Assume now that $x \in (\mathbb{R}^N \setminus B) \cap \Sigma_a^+$, i.e., $|x| \geq 1$ and $x_1 \geq a$. 
	Then we have
	$$
	|\bar{x}|^2 
	= 
	|x|^2 + 4a(a-x_1)
	\geq
	|x|^2 + 4a^2 - 4a|x| 
	= 
	(|x|-2a)^2
	> 
	r^2, 
	$$
	since $a \in (0, \frac{1-r}{2})$.
	Recalling that $u \geq 0$ in $A_{r,1}$, we conclude that $u(\bar{x}) \geq 0$.
	This implies that $(P_a u)(x) = \min\{u(x), u(\bar{x})\} = 0$.
	Finally, combining both cases, we deduce that $(P_a u)(x) =  0$ for any $x \in \mathbb{R}^N \setminus B$ provided $a \in (0, \frac{1-r}{2})$, and hence $P_a u \in X_0^s(B)$, cf.\ Figure \ref{fig:2}. 
	
	Taking now any $a \in (0, \frac{1-r}{2})$ and recalling that $u$ satisfies the equalities \eqref{eq:-1}, we use the inequalities \eqref{eq:5} of Lemma \ref{lem:2} together with the equality in \eqref{eq:3} in order to deduce that $P_a u$ satisfies the assumptions \eqref{eq:1}  of Lemma \ref{lem:1}.
	That is, $P_a u$ is also a second eigenfunction of the problem \eqref{eq:D}. 
	Notice that $P_a u$ is rotationally invariant with respect to the $x_1$-axis, as it follows from the definition of polarization and the radial symmetry of $u$.
	However, $P_a u$ is not radially symmetric at least if $a \in (0, \frac{1-\rho}{2})$.
	Indeed, recalling that $u>0$ in $A_{\rho,1}$ and $u=0$ on $\partial B_\rho$, and taking $x^* = (\rho + 2a,0,\dots,0) \in A_{\rho,1} \cap \Sigma_a^+$, we see that
	\begin{equation}\label{eq:nonradP}
	(P_a u)(x^*) = \min\{u(x^*), u(\overline{x^*})\} = 0,
	\quad
	(P_a u)(-x^{*}) = \max\{u(-x^{*}), u(\overline{-x^{*}})\} > 0,
	\end{equation}
	and so $P_a u$ is nonradial.
	
	From now on, we assume that $a \in (0, \frac{1-\rho}{2})$ is fixed.
	Let us prove that $P_a u$ \textit{cannot} be a second eigenfunction for such $a$.
	By Proposition~\ref{prop:spectrum} (see the discussion thereafter), we have to distinguish two cases:
	
	\textit{Case I.} 
	$\lambda_2 = \lambda_\circledcirc < \lambda_\ominus$. 
	In this case, there exists $k \geq 1$ such that
	$$
	\lambda_2 = \lambda_{N,2} = \dots = \lambda_{N,k+1} < \min\{\lambda_{N,k+2},\lambda_{N+2,1}\}.
	$$
	That is, the eigenspace $ES(\lambda_2)$ of  $\lambda_2$ is generated by $k$ linearly independent radial eigenfunctions, which implies that any second eigenfunction is radial. 
	But this is impossible since the second eigenfunction $P_a u$ is nonradial.

	\textit{Case II.} $\lambda_2 = \lambda_\circledcirc = \lambda_\ominus$.
	In this case, there exists $k \geq 1$ such that
	\begin{equation*}\label{eq:caseII}
		\lambda_2 = \lambda_{N,2} = \dots = \lambda_{N,k+1} 
		=
		\lambda_{N+2,1} < 	
		\min\{\lambda_{N,k+2},\lambda_{N+2,2}, 
		\lambda_{N+4,1}\}.
	\end{equation*}
	Since $\lambda_{N+2,1}$ is counted in the spectrum of \eqref{eq:D} with multiplicity $N$, we conclude that 
	the multiplicity of $\lambda_2$ is $k+N$, and 
	$$
	ES(\lambda_2) = \text{span}\{u_1,\dots,u_k, \xi_1,\dots,\xi_N\},
	$$ 
	where each $u_i$ is a radial eigenfunction, and each $\xi_i = x_i \varphi_{N+2,1}(|x|)$ is the anti-symmetric eigenfunction associated with $\lambda_\ominus$. 
	In particular, recalling that $\varphi_{N+2,1}(|x|)>0$ for $|x| \in [0,1)$, we see that the nodal set of $\xi_i$ is the intersection of the ball $B$ with the hyperplane $\{x_i=0\}$.

		\begin{figure}[h]
		\hspace*{\fill}
		\begin{minipage}[t]{0.48\linewidth}
			\centering
			\vspace{0pt}
			\includegraphics[width=0.75\linewidth]{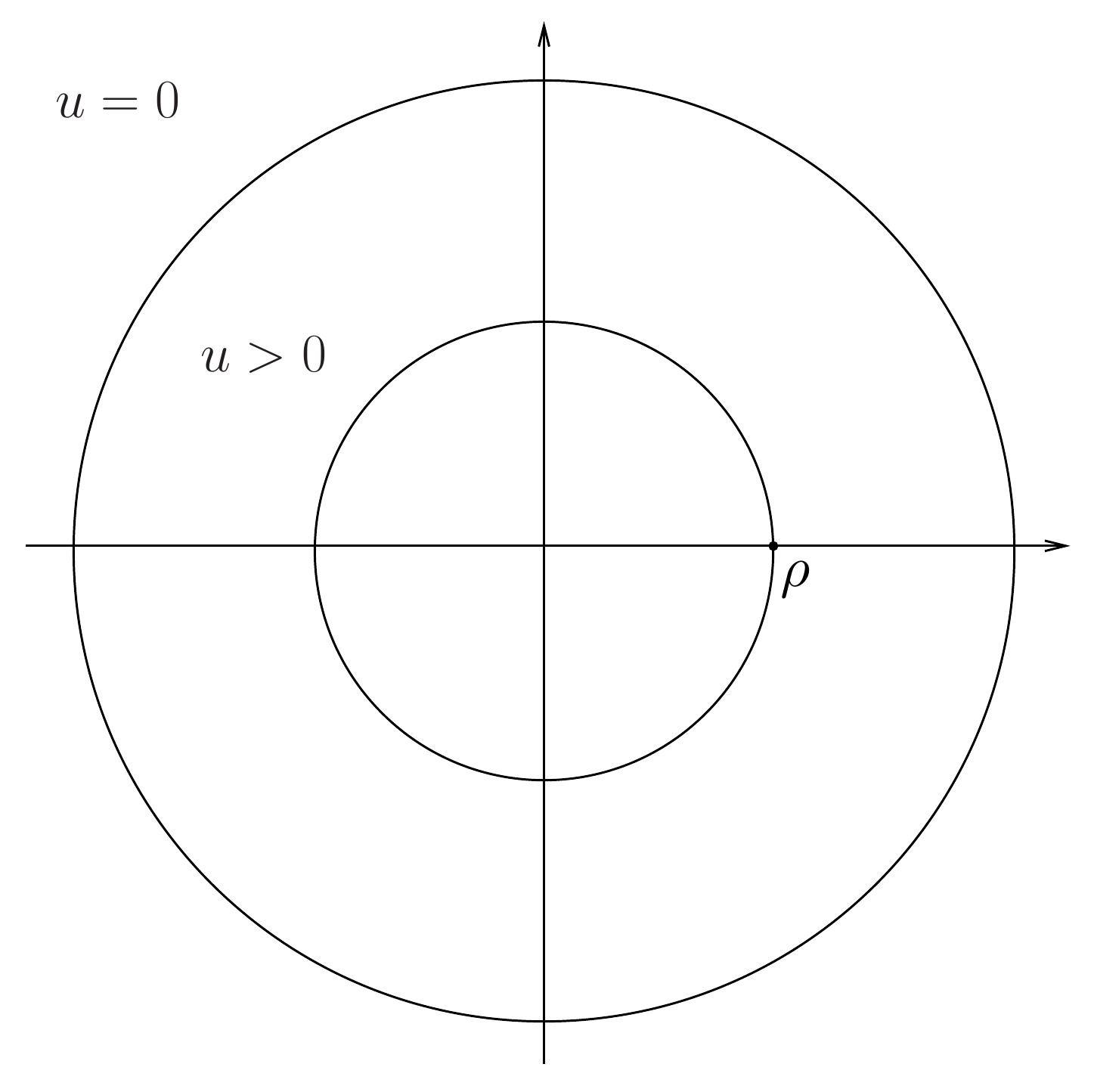}
			\caption{A radial eigenfunction $u$.}
			\label{fig:1}
		\end{minipage}
		\hfill
		\begin{minipage}[t]{0.48\linewidth}
			\centering
			\vspace{0pt}
			\includegraphics[width=0.75\linewidth]{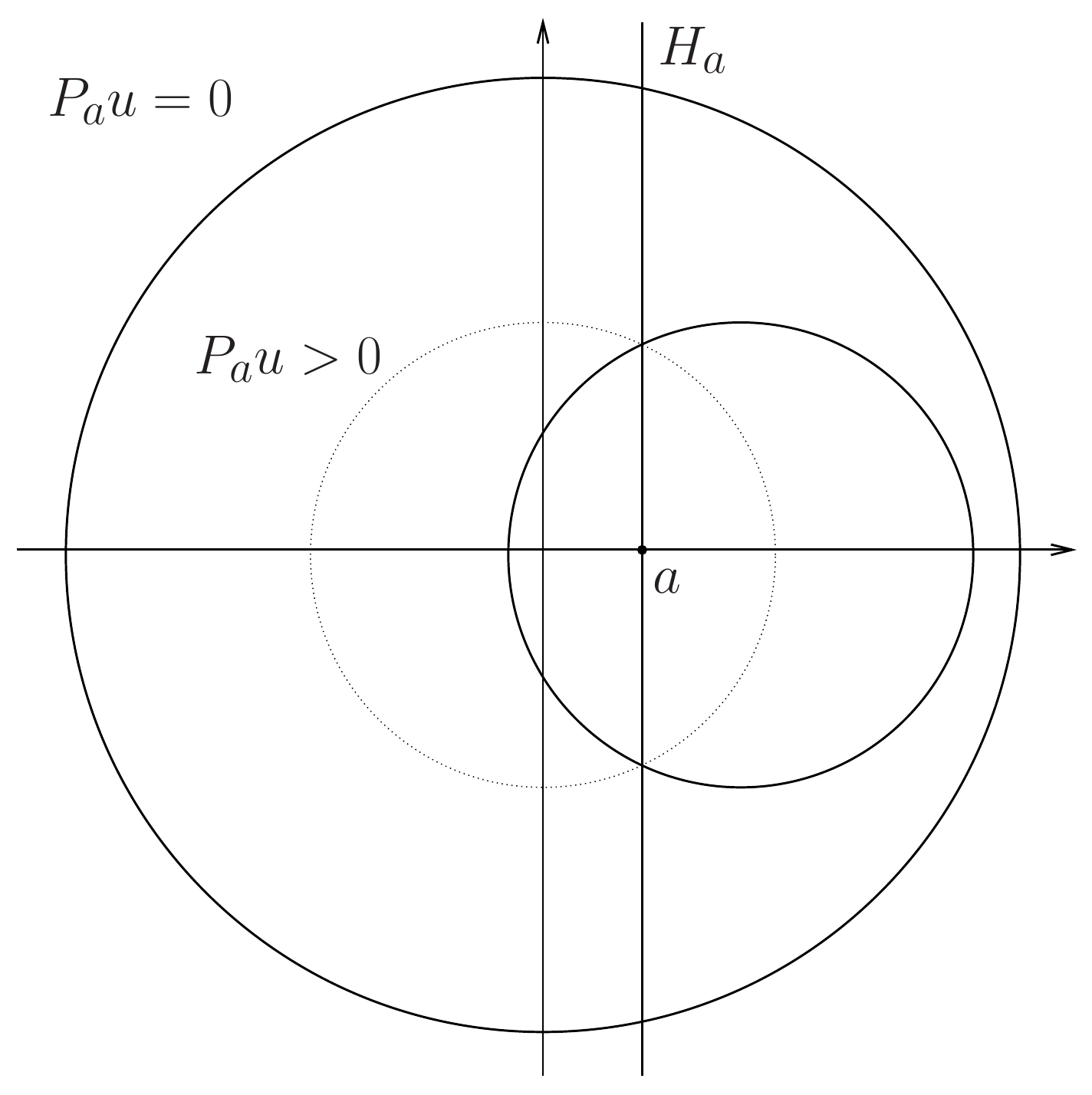}
			\caption{$P_{a} u$ for $a \in (0, \frac{1-r}{2})$.}
			\label{fig:2}
			\hspace*{\fill}
		\end{minipage}
	\end{figure}
	
	In order to show that \textit{Case II} is also impossible, let us prove that $P_a u \not\in ES(\lambda_2)$.
	Suppose this assertion is false,
	i.e., there exist $v \in \text{span}\{u_1,\dots,u_k\}$, $c \in \mathbb{R}$, and $\eta \in \text{span}\{\xi_2,\dots,\xi_N\}$ such that
	$$
	(P_a u)(x)
		=
		v(x) + c \, \xi_1(x) + \eta(x)
		\quad \text{for any}~
		x \in B.
	$$
	Noting that $P_a u$, $v$, and $\xi_1$ are rotationally invariant with respect to the $x_1$-axis, while any element of $\text{span}\{\xi_2,\dots,\xi_N\}\setminus \{0\}$ does not have such symmetry, we conclude that $\eta \equiv 0$ in $B$. 
	On the other hand, $P_a u$ is nonradial and, moreover, $P_a u \not\in \text{span} \{\xi_1\}$ in view of \eqref{eq:nonradP} and the anti-symmetry of $\xi_1$ with respect to the hyperplane $\{x_1=0\}$.
	These two facts enforce $v \not\equiv 0$ in $B$ and $c \neq 0$, and, consequently, we have the decomposition
	\begin{equation}\label{eq:w-decompos}
		(P_a u)(x)
		=
		v(x) + c \, \xi_1(x) 	
		\quad \text{for any}~
		x \in B.
	\end{equation}
	
	Thanks to the fact that the choice of $u$ on p.~\pageref{choice-of-u} (which is the ``generating'' radial second eigenfunction for $P_a u$) was  arbitrary, let us show that $u$ can be chosen in such a way that $r(u) = \widetilde{r}:=\inf r(\zeta)$,  
	where the infimum is taken among all $\zeta \in \text{span}\{u_1,\dots,u_k\} \setminus \{0\}$.
	Noting that $r(t \zeta)=r(\zeta)$ for any $\zeta \in \text{span}\{u_1,\dots,u_k\} \setminus \{0\}$ and any $t>0$,
	we have
	$$
	\widetilde{r}
	=
	\inf 
	\left\{
	r\left(\sum_{i=1}^{k} c_i u_i\right):~ 
	(c_1,\dots,c_k) \in \mathbb{S}^{k-1}
	\right\},
	$$
	where $\mathbb{S}^{k-1}$ stands for the unit $(k-1)$-dimensional sphere in $\mathbb{R}^k$.
	We prove that $\widetilde{r}$ is attained.
	Let $\{(c_{1,n},\dots,c_{k,n})\}_{n \in \mathbb{N}} \subset \mathbb{S}^{k-1}$ be a minimizing sequence for $\widetilde{r}$. 
	Clearly, there exists $(c_1^*,\dots,c_k^*) \in \mathbb{S}^{k-1}$ such that $c_{i,n} \to c_i^*$ as $n \to +\infty$ for any $i \in \{1,\dots,k\}$, up to a subsequence.
	Let us denote $\zeta_n = \sum_{i=1}^{k} c_{i,n} u_i$ and $u^* = \sum_{i=1}^{k} c_i^* u_i$, and show that $r(u^*) = \widetilde{r}$, that is, $u^*$ is a minimizer for $\widetilde{r}$.
	Evidently, the case $r(u^*) < \widetilde{r}$ is impossible due to the definition of $\widetilde{r}$. 
	Thus, suppose, by contradiction, that $r(u^*) > \widetilde{r}$. 
	By the definition \eqref{eq:r} of $r(u^*)$, there exists $\varrho \in (\widetilde{r}, r(u^*))$ such that $u^*(x)<0$ for any $x$ satisfying $|x| = \varrho$.
	On the other hand, since $r(\zeta_n) \to \widetilde{r}$ as $n \to +\infty$, there exists $n_0 \in \mathbb{N}$ such that $r(\zeta_n) < \varrho$ for all $n \geq n_0$, and hence $\zeta_n(x) \geq 0$ for all $n \geq n_0$ and all $x$ satisfying $|x| = \varrho$.
	Thanks to the fact that every $u_i \in L^\infty(B)$ (see \cite[Proposition 4]{Servadei3}), the convergence $c_{i,n} \to c_i^*$ yields
	$$
	|u^*(x) - \zeta_n(x)| 
	=
	|\sum_{i=1}^k c_{i,n} u_i(x) - \sum_{i=1}^k c_i^* u_i(x)| 
	\leq
	\max_{i=1,\dots,k} \|u_i\|_{L^\infty(B)}
	\sum_{i=1}^k |c_{i,n}-c_i^*|  
	\to 0.
	$$
	This implies that $u^*(x) \geq 0$ for any $x$ satisfying $|x| = \varrho$, which is a contradiction.
	Therefore, $\widetilde{r}$ is attained by $u^*$.
	Notice that $u^* \not\equiv 0$, and $\widetilde{r}>0$ since $u^*$ is sign-changing.
	Hereinafter, we assume that $P_a u$ is generated by $u^*$ and we denote $u=u^*$.
	
	Since the eigenfunction $P_a u$ has the decomposition \eqref{eq:w-decompos} and $\xi_1=0$ on $\{x_1=0\}$, we have $P_a u=v$ on $\{x_1=0\}$.
	Thus, 
	we see from the definition \eqref{eq:r} of $r(v)$ that
	$(P_a u)(x) \geq 0$ for any $x$ such that $x_1=0$ and $|x| \geq r(v)$, and for any $\varepsilon>0$ there exists $x^\# $ such that $x_1^\#=0$, $|x^\#| \in [r(v)-\varepsilon,r(v))$, and $(P_a u)(x^\#)<0$.
	Clearly, $\overline{x^\#} \in \Sigma_a^+$ since $a>0$.
	Let us show that $\overline{x^\#} \not\in B_{r(u)}$ provided $\varepsilon < 2a^2$.
	Indeed, since $x_1^\#=0$, and $r(v) \geq r(u)$ by the choice of $u$, we have
	\begin{align*}
		|\overline{x^\#}|^2 
		= 
		(2a)^2 + (x_2^\#)^2 + \dots + (x_N^\#)^2
		=
		(2a)^2 + |x^\#|^2 
		&\geq
		(2a)^2 + (r(v) - \varepsilon)^2
		\\
		&\geq (2a)^2 -2\varepsilon + \varepsilon^2 + r(v)^2
		> r(u)^2.
	\end{align*}
	This implies that $u(\overline{x^\#}) \geq 0$ whenever $\varepsilon \in (0,2a^2)$, which leads to the following contradiction:
	$$
	0>(P_a u)(x^\#) = \max\{u(x^\#), u(\overline{x^\#})\} \geq 0.
	$$
	Therefore, \textit{Case II} is impossible.
	
	By excluding \textit{Cases I} and \textit{II}, we complete the proof of Theorem \ref{thm}.
	\qed

	\medskip
	\subsection*{Acknowledgements}
    The authors are grateful to the anonymous referee whose valuable comments and suggestions helped to improve and clarify the manuscript.
    In particular, his suggestions helped to strengthen the statement of Lemma \ref{lem:1} which led to certain simplifications in the proof of Theorem \ref{thm}.

\end{document}